\documentclass[11pt,reqno]{amsart} %change back to 11pt
\usepackage{amssymb,latexsym}
\usepackage{graphicx}
\usepackage{fancyhdr,amssymb}
\usepackage{url}		%does nice formatting of URLs

\theoremstyle{plain}
\numberwithin{equation}{section}
\newtheorem{thm}{Theorem}[section]
\newtheorem{theorem}[thm]{Theorem}

\newtheorem{example}[thm]{Example}
\newtheorem{definition}[thm]{Definition}
\newtheorem{proposition}[thm]{Proposition}
\newtheorem{corollary}[thm]{Corollary}
\newtheorem{comment}[thm]{Comment}

\begin{document}
\fancyhead{}
\renewcommand{\headrulewidth}{0pt}
\fancyfoot{}
\fancyfoot[LE,RO]{\medskip \thepage}
\fancyfoot[LO]{\medskip DECEMBER 2020}
\fancyfoot[RE]{\medskip VOLUME 58, NUMBER 5 }

\setcounter{page}{1}

\title[Recursive Triangles Appearing Embedded in Recursive Families]
{Recursive Triangles Appearing Embedded in Recursive Families}
\author{Russell Jay Hendel}
\address{Department of Mathematics, Towson University}
\email{rhendel@towson.edu}
 
\begin{abstract} 
 
We continue the work begun in OEIS sequence A332636 which presents recursive sequences that have  triangles that appear embedded in them.
 This paper i) generalizes the main result presented in A332636, ii)  provides a complete set of definitions
and underlying concepts, and iii) provides a complete proof.  
\end{abstract}

\maketitle

\section{ Illustrative Examples and Introduction}

This introductory section presents illustrative examples of a triangle  embedded in a recursive sequence, a
concept introduced  in \cite{1}.   This paper i) generalizes the main result presented in \cite{1}, ii)  provides a complete set of definitions
and underlying concepts, and iii) provides a complete proof. This paper is self-contained; no familiarity
with \cite{1} is assumed or needed.

We first provide some needed prerequisites and conventions.

Here and throughout the paper, we deviate from the textbook custom of having the leading coefficient
of a characteristic polynomial equal to one. Instead, we let the leading coefficient be minus one;
as a consequence, the coefficients of the characteristic polynomial  are identical with the coefficients
on the right-hand-side  of the corresponding recursion.

Here and throughout the paper, recursive sequences will be represented with either  the letter $G,$  or by  $\{G_i\}_{i \ge 1};$
characteristic polynomials will be represented by $p_k(X),$ where $k$ is the order of the corresponding recursion.

Here and throughout the paper, given a recursion of order $k,$ the initial
values are
\begin{equation}\label{ivalues}
	G_1 = 1, \qquad G_i = 0, \qquad 2 \le i \le k.
\end{equation}
As usual however, a  recursive sequence generated by a recursion with constant coefficients
 may be made doubly infinite. However, this will not be needed in the sequel.

Here and throughout the paper if $r, q$ are given positive integers with 
\begin{equation}\label{qge2}
	q \ge 2,
\end{equation}
we define 
\begin{equation}\label{kr}
	k=k(r) = 1 + rq.
\end{equation}

For the illustrative example presented in this section, we let $ q = 3.$ We consider the following recursions and associated characteristic polynomials
of orders $k(r), r=1,2,3.$ 
\begin{align}\label{4710}
 	\; & G_n = G_{n-4} - G_{n-3} - 25 G_{n-2} - G_{n-1}, &  p_4(X) = 2 - \frac{X^5-1}{X-1} - 24 X^2,   \nonumber \\
	\; & G_n = G_{n-7} - G_{n-6} - 25 G_{n-5} -\sum_{i=1}^4 G_{n-i}, &   p_7(X) = 2 - \frac{X^8-1}{X-1} - 24 X^2,  \\
	\; & G_n = G_{n-10} - G_{n-9} - 25 G_{n-8} -\sum_{i=1}^7 G_{n-i}, & p_{10}(X) = 2 - \frac{X^{11}-1}{X-1} - 24 X^2. \nonumber
\end{align}

The triangles that appear embedded in these sequences are found in rectangular arrangements of consecutive sequence members.
To describe these rectangles we must indicate i) where in the recursive sequences 
these consecutive members begin, ii) the number of rows, and iii) the number of columns involved.  
 
Given positive integers $r, q,$ (with $q$ satisfying \eqref{qge2}),  define $c(r),$ the number of columns, 
for the rectangle containing the triangle appearing embedded in the recursive sequence of order $k(r),$
  by 
\begin{equation}\label{cr}
	c=c(r) = 2+ (r-1)q.
\end{equation}

Given the recursions and initial values in  \eqref{4710} and \eqref{ivalues} respectively,
and using \eqref{kr} and \eqref{cr}, it is routine  to calculate
$G_{k(r)+1},  G_{k(r)+2}, \dotsc, G_{k(r)+rc(r)}.$ Tables 1-3 present these sequence members laid out as $r$ rows of $c(r)$ columns.
The row and column indices for the rectangle entries are  
\begin{equation}\label{rectangle}
	\langle G_{k(r)+1},  G_{k(r)+2}, \dotsc, G_{k(r)+r c(r)} \rangle = \langle R_{1,1}, R_{1,2}, \dotsc, R_{r, c(r)}  \rangle.
\end{equation}

\begin{center}
\begin{table}
\begin{tabular}{||c|c c||}
\hline \hline
$\text{Position}$	&	1	&	2 	\\
\hline							
$G_5=R_{1,1}, G_6=R_{1,2}$		&	1	&	-1	\\
\hline \hline
\end{tabular}
\begin{caption}{ $1 \times 2$ rectangle for the order $k(1)=4$ recursion} \end{caption}
\end{table}
\end{center}

\begin{center}
\begin{table}
\begin{tabular}{||c|c c c c c||}
\hline \hline
$\text{Position}$		&	1	&	2	&	3	&	4	& 	5	\\
\hline
$R_{1,1}, \dotsc, R_{1,c}$	&	1	&	-1	&	0	&	0	&	0	\\
$R_{2,1}, \dotsc, R_{2,c}$	&	-24	&	48	&	-22	&	-3	&	1	\\
\hline \hline
\end{tabular}
\begin{caption}{ $2 \times 5$ rectangle for the order $k(2)=7$ recursion} \end{caption}
\end{table}
\end{center}

\begin{center}
\begin{table}
\begin{tabular}{||c|c c c c c c c c||}
\hline \hline
$\text{Position}$		&	1	&	2	&	3	&	4	& 	5	&	6 &	7 &	8\\
\hline
$R_{1,1}, \dotsc, R_{1,c}$	&	1	&	-1	&	0	&	0	&	0	& 	0 & 	0 & 	0\\
$R_{2,1}, \dotsc, R_{2,c}$	&	-24	&	48	&	-22	&	-3	&	1	&	0 & 	0 & 	0\\
$R_{3,1}, \dotsc, R_{3,c}$	&	576	&	-1728	&	1632	&	-336	&	-188	&	40 & 	5 & 	-1\\
\hline \hline
\end{tabular}
\begin{caption}{ $3 \times 8$ rectangle for the order $k(3)=10
$ recursion} \end{caption}
\end{table}
\end{center}
	
These tables nicely illustrate the idea of a triangle  appearing embedded in these recursive sequences. 
The triangles associated with different recursions show compatibility; for example,
the two triangle  rows of the order-7 sequence are also the first two triangle  rows of the order-10 sequence with extra zeroes.

Throughout the paper, we will abuse language and refer to configurations similar to those in  Tables 1 - 3 as triangles or rectangles. This should cause
no confusion since the term rectangle refers to the entire rectangular array, while the term triangle  refers to the collection of rows
with the zeroes on the right side of these rows omitted.  
 
To present the complete definition of a triangle  appearing embedded in a recursive sequence, we need
one more ingredient. To motivate this ingredient  notice that we can always take any $rc(r)$ consecutive members of a recursive sequence and arrange them as a
rectangle. A key point in the above tables is that the triangle's right-hand side is delimited by a sequence of zeroes which terminate the rows (except the last). 
We therefore  introduce the function $l(t) $ equaling the rectangle position of the last non-zero element of the $t$-th row.  Given an $r$ and using \eqref{rectangle}, $l(t)$ must satisfy
\begin{equation}\label{lt}
	R_{t,l(t)} \neq 0,   1 \le t \le r,
\end{equation}
and 
\begin{equation}\label{lt2}
	  R_{t,u} = 0, l(t)<u \le c(r), 1 \le t \le r.
\end{equation}
Note that $l(t)$ is well defined since $l(t) = c(r)$ if no element in the $t-th$ row is zero.

We summarize the preceding discussion with the following definition.

\begin{definition} With notations as above, we say that  a recursive sequence   $\{G_i\}_{i\ge 0},$ of order $k=k(r),$
for some positive integer $r,$ has  a triangle that  appears embedded in the
sequence (at $G_{k+1}, \dotsc, G_{k+rc(r)}$) if \eqref{lt} and \eqref{lt2} hold and additionally
the sequence $l(1), l(2), \dotsc, l(r)$ is strictly increasing. 
Two sequences of orders $k_1=k(r_1)$ and $k_2=k(r_2)$ with   triangles appearing embedded in them are compatible  
 if $l_1(t) = l_2(t)=l(t), 1 \le t \le min(r_1,r_2),$ 
and $R_{t,s}^{(k_1)} = R_{t,s}^{(k_2)}$ for
$ 1 \le s \le l(t), 1 \le t \le min(r_1,r_2)$ 
(here, $R^{(k_i)}, i=1,2$ alnd $l_i(t), i=1,2$ are the rectangles 
of  the order-$k_1$ and order-$k_2$ recursions respectively).  A family of recursive sequences appears to have a triangle  embedded in it if
all members of the family have triangles appearing embedded in them and every two members of the family are compatible. 
\end{definition}

To complete our introductory definitions we need  a way to talk about a family of recursive sequences.
The characteristic polynomials presented in \eqref{4710} nicely motivate the idea of using a Taylor series to indicate a family of recursive sequences.

\begin{definition} We associate to every  Taylor series, $T(X),$ a family of recursive sequences where the $k$-th member
of the recursive family satisfies the recursion corresponding to the $k$-th approximating Taylor polynomial regarded as a characteristic
polynomial, with initial values given by \eqref{ivalues}.\end{definition}

\begin{example} Let $T(X) = 2 - \frac{1}{1-X} - 24X^2.$ Then the orders 4,7, and 10 members of the associated family of recursive sequences
are the sequences generated by the initial values \eqref{ivalues} and the 
recursions presented in \eqref{4710}. \end{example}

We close this section by pointing out that the triangles that appear embedded in the recursive sequences   satisfy
a triangle recursion.  

\begin{proposition} If a doubly infinite recursive sequence 
satisfies a recursion of order $m$ whose characteristic polynomial is $p_m(X),$ then it also
satisfies the recursion associated with $(X-1)p_m(X),$ regarded as a characteristic polynomial.
\end{proposition}
\begin{proof}
Interpret $X$ as the backward shift operator, that is, for any integer $i,$ $X(G_i) = G_{i-1}.$ Then since $p_m(X)$ regarded
as an operator annihilates any recursive sequence which it satisfies, therefore, $(X-1)p_m(X)$ also annihilates this sequence, because
operators are associative. Therefore, the sequence   also satisfies the recursion corresponding to $(X-1)p_m(X).$
\end{proof}

\begin{comment} The proposition can clearly be generalized with $X-1$ replaced by any polynomial with integer coefficients. However,
we will not need this in the sequel. 
\end{comment}

It follows that the doubly infinite sequences satisfying the three recursions presented in \eqref{4710} and satisfying \eqref{ivalues}\
also satisfy the recursions associated with $(X-1)p_{k(r)}(X).$ The resulting corresponding characteristic polynomials and triangle  recursions are as follows. 

\tiny
\begin{align}\label{T4710}
	\; & (X-1)p_4(X) =    2(X-1) - X^5+1 -24X^3 + 24X^2; & \qquad R_{t+1,s} = -24 R_{t,s} + 24 R_{t,s-1} + 2 R_{t,s-2} - R_{t,s-3} \nonumber \\   
	\; & (X-1)p_7(X) =    2(X-1) - X^8+1 -24X^3 + 24X^2; & \qquad R_{t+1,s} = -24 R_{t,s} + 24 R_{t,s-1} + 2 R_{t,s-2} - R_{t,s-3}  \\ 
	\; & (X-1)p_{10}(X) = 2(X-1) - X^{11}+1 -24X^3 + 24X^2; & \qquad R_{t+1,s} = -24 R_{t,s} + 24 R_{t,s-1} + 2 R_{t,s-2} - R_{t,s-3}. \nonumber   
\end{align}
\normalsize

The reader can easily verify that the rectangles presented in Tables 1 - 3 satisfy the recursions given in 
 \eqref{T4710} for $1 \le t \le r, 1 \le s \le c(r),$ with
the obvious boundary conditions 
\begin{equation}\label{boundary}
	R_{0,u} = 0, \qquad 1 \le u \le c(r), \qquad R_{t,1-s} = R_{t-1,c(r)+1-s}, 2 \le t \le r, 1 \le s \le c(r).  
\end{equation} 
The conditions $R_{0,u}=0, 1 \le u \le c(r)$ are consistent with \eqref{ivalues} and with the fact that by \eqref{qge2}, \eqref{kr} and \eqref{cr},
$c(r)  \le  k(r)-1$ assuring that the set $\{G_2,\dotsc, G_{k(r)}\}$ has as least $c(r)$ zero values. 

\section{The Main Theorem}

This section presents the main theorem which fully generalizes the examples of the preceding section.
 The remaining sections of the paper will prove this
theorem.

\begin{theorem} Let  $q$ satisfying \eqref{qge2} and $a_i$  be integers satisfying
\begin{equation}\label{qap}
	  	a_i \ge 1, 2 \le i \le q, 		\qquad  		a_q \ge 2.
\end{equation}
For positive integer $r$ define $k(r)$ and $c(r)$ by \eqref{kr} and \eqref{cr} respectively.
Consider the sub-family of recursive sequences of orders $k(r), r=1,2,3, \dotsc,$  associated with the Taylor Series
\begin{equation}\label{TX}
	T(X) = 2 - \frac{1}{1-X} - \displaystyle \sum_{i=2}^q  (a_i-1) X^{ i -1}.
\end{equation}
Then  with the notation of   \eqref{rectangle}, we have \\
i)
\begin{equation}\label{firstrow}
	G_{k(r)+1}=R_{1,1} =1, G_{k(r)+2}=R_{1,2} = -1, G_{k(r)+u} = R_{1,u} = 0, 3 \le u \le c(r);
\end{equation} \\
ii) The triangle  entries $R_{i,j}, 2 \le i \le r, 1 \le j \le c(r),$ satisfy both the  underlying recursion of order $k(r)$ (which in the
sequel we will call the $G$ recursion)  and the following
triangle  recursion  (which in the sequel we will call the $T$ recursion)
\begin{equation}\label{trianglerecursion}
	R_{i,j} = -(a_q-1) R_{i-1,j} - \biggl(  (a_{q-2}-a_{q-1}) R_{i-1,j-1} + \dotsc \biggr) +(1 + a_2  ) R_{i-1,j-q+1} - R_{i-1,j-q} 
\end{equation} 
  with the boundary conditions given by \eqref{boundary}; \\ 
iii) For $1 \le t \le r,$
\begin{equation}\label{rt1neq0}
	R_{t,1} \neq 0;
\end{equation}
iv) equations \eqref{lt} and \eqref{lt2} are satisfied for $1 \le t \le r$
by the function 
\begin{equation}\label{ltkr}
	l(t) = 2+ (t-1)q;  
\end{equation}
  in other words, \eqref{ltkr} gives the correct functional form for
 identifying the last non-zero element in each row; \\
v) this family of recursive sequences appears to have a triangle embedded in it. 
\end{theorem}

\begin{comment} 
By Definition 1.1,  i)-iv) implies (v). We will devote one section each to the proofs
of i) and  ii) and then one section for the proof of both (iii) and (iv).
\end{comment}

This main theorem has been amply illustrated in Section 1. 
The  main theorem generalizes the result in [1] since the examples in [1]
correspond to the special case where only one of the $a_i, 2 \le i \le q,$ is greater than 1, 
while  the main theorem   allows 
several $a_i$ to be greater than 1. 

\section{The First Row}

The purpose of this section is to prove \eqref{firstrow}.  
To proceed with the proof we fix an integer $r;$  
 by the conventions of \eqref{kr} and \eqref{cr}, $k$ and $c$ refer to $k(r)$ and $c(r)$ respectively

To compute values of $G$ we need the underlying recursion which by Definition 1.2 is the recursion associated with its 
characteristic polynomial, $p_{k}(X),$ which in turn is the approximating Taylor polynomial (of degree $k$) to the Taylor series, \eqref{TX}.
Therefore, the characteristic polynomial is given by 
\begin{equation}\label{cp}
	p_{k}(X) = 2 - \frac{X^{k+1}-1}{X-1} -  \sum_{i=2}^q  (a_i-1) X^{ i -1},
\end{equation}
implying that the recursion is given by 
\begin{equation}\label{temp1}
	G_n = G_{n-k} - \biggl(  a_2 G_{n-k+1} + \dotsc + a_q G_{n-k+q-1} \biggr) - \biggl( G_{n-k+q} + G_{n-k+q+1} \dotsc +G_{n-1}\biggr).
\end{equation}
By \eqref{kr}, $k>q$ and in fact  

\begin{equation}\label{temp2}
		\text{The coefficients of the last $k-q = (r-1)q+1$ summands on the right-hand side of \eqref{temp1} are -1}.
\end{equation}

We prove \eqref{firstrow} by computing each of the values $G_{k+u}, 1 \le u \le c(r).$

\textbf{The value of $G_{k+1}.$} Letting $n=k+1$ in \eqref{temp1} shows $G_{k+1} = 1$  since by \eqref{ivalues}, $G_{n-k}=G_1=1,$ but
$G_i=0, 2 \le i \le k.$

\textbf{The value of $G_{k+2}.$} Letting $n=k+2$ in \eqref{temp1} and noting that, by \eqref{temp2}, the coefficient of $G_{n-1}$ in 
\eqref{temp1} is -1,
we see that $G_{k+2} = -1,$ since by \eqref{ivalues} $G_i=0, 2 \le i \le k.$

If $r=1,$ then by \eqref{cr}, $c=2$ and we have completed the proof of \eqref{firstrow}. Therefore, for
the rest of the proof we assume $r \ge 2.$ 

\textbf{The value of $G_{k+3}$}. Let $n=k+3$ in \eqref{temp1}. By our assumption on $r$ and \eqref{temp2}, the last two coefficients
in \eqref{temp1} are minus 1. Therefore $G_{k+3} = - G_{n-2}-G_{n-1} = -(G_{k+2} + G_{k+1}) = 0 $ by \eqref{ivalues} and
by the results we just proved for $G_{k+1}, G_{k+2}. $

\textbf{The value of $G_{k+u}, 4 \le u \le c.$}  Using an induction assumption, assume $G_{k+u} = 0, 3 \le u \le v-1 \le c-1,$ the base case 
when $v=4$ having just been proven. We proceed to prove $G_{k+v}=0.$ Let $n=k+v$ in \eqref{temp1}. $G_{k+v}=0$ iff both 
 $G_{n-(v-1)}=G_{k+1}$ and  $G_{n-(v-2)}=G_{k+2}$ have a coefficient of minus 1 in \eqref{temp1} since then by \eqref{ivalues} and our
induction assumption $G_{k+v} = -G_{n-(v+2)} - G_{n-(v+1)} = -(G_{k+1}+G_{k+2}) = 0.$ The proof is therefore completed by \eqref{temp2}.

This completes the proof of \eqref{firstrow}.

\section{The Triangle  Recursion}

In this section, we prove \eqref{trianglerecursion}.

Fix  $r \ge 2$. The characteristic polynomial for the member of the recursive family of degree k is given by \eqref{cp}. 

By Proposition 1.4, the recursive sequence of this family member also satisfies 
\begin{equation}\label{temp32}
	(X-1)p_k(X) = 2(X-1) - (X^{k+1}-1) - (X-1)  \displaystyle \sum_{i=2}^q (a_i-1) X^{i-1}.
\end{equation} 
In the sequel, we will speak about  the three summands in \eqref{temp32} with the understanding that  
the first summand refers to $2(X-1),$
the second summand refers to the parenthetical expression, $(X^{k+1}-1) $ 
and  the third summand refers to the product of $(X-1)$ with the sum.

We  now calculate the coefficients of  $X^u, 0 \le u \le k+1,$ in \eqref{temp32}. 
\begin{itemize}
\item Clearly, the highest exponent occurring in   \eqref{temp32} is {k+1}; the coefficient of $X^{k+1}$ is -1, since 
by \eqref{kr}, $q < k-1$ and therefore only the second summand
contributes to this coefficient.
\item The second biggest exponent occurring in \eqref{temp32} is $q;$ the coefficient of $X^q$ is $-(a_q-1).$
\item The coefficient of $X^0$ is -1 since the contributions from the three summands are -2, 1, and 0 respectively.
\item The coefficient of $X^1$ is $1+a_2$ since the first and third summands contribute 2 and $a_2 - 1$ respectively.
\item The coefficient of $X^e,$ for $2 \le e \le q-1,$ is $-(a_e - a_{e+1})$ since only the third summand contributes. 
\end{itemize}

Therefore, the recursion corresponding to $(X-1)p_k(X)$ is
\begin{equation}\label{temp33}
	G_n = -(a_q-1) G_{n-(k+1-q)} - \biggl(( a_{q-1} - a_q) G_{n-(k+1-(q-1))} \dotsc \biggr) + (1+a_2) G_{n-(k+1-1)} - G_{n-(k+1)}
\end{equation}

To complete the proof of \eqref{trianglerecursion}, we must convert \eqref{temp33} into \eqref{trianglerecursion}.
First observe, that by \eqref{kr} and \eqref{cr}
$$
	k + 1 - q = (1+rq)+1-q = 2+(r-1)q = c.
$$
It immediately follows that 
\begin{equation}\label{temp35}
	G_{n-(k+1-q)} = G_{n-c}.
\end{equation}
Equation \eqref{temp35} motivates the idea of representing the linear sequence by a rectangular array, since to compute $G_n$ instead of
going back $k+1-q$ columns one only need go up one row provided the row lengths are $c.$ 
This motivation will be made precise in the corollary below.

First however,  we complete the proof. Let $n=k+u$ for some $u,  c+1 \le u \le rc.$ Then by \eqref{rectangle}
$$
	G_n = R_{t,s}, \qquad \text{for some $t,s$ with $2 \le t \le r, 1 \le s \le c$}
$$
But then, using \eqref{temp35}, 
\begin{equation}\label{temp36}
	G_{n-(k+1-q)} = R_{t-1,s},
\end{equation}
since the lengths of all rows are $c.$
Equation \eqref{temp36} implies
$$
	G_{n-(k+1-q)-u} = R_{t-1,s-u}, \qquad 1 \le u \le q,
$$
and this completes the proof of \eqref{trianglerecursion}.
 
\begin{corollary} Computation of $G_n$ requires: 
\begin{itemize}
 \item A lookback of $k$ columns and $k$ multiplications using the $G$ recursion, \eqref{temp1};  
 \item A lookback of $q$ columns one row up and $q+1$ multiplications using the $T$ recursion, \eqref{trianglerecursion}.
\end{itemize}
\end{corollary}

A key point of this corollary is that $k$ is going to infinity while $q$ is constant.

\section{Completion of Proof of the main theorem}

In this section, we prove (iii) and (iv) of the main theorem. 
We first review what is given and what has to be proven.
We are given a positive integer $q$ satisfying \eqref{qge2},
coefficients $a_i, 2 \le i \le q,$  satisfying \eqref{qap}, and  a power series \eqref{TX}. 
Using $q,$ we define $k$ and $c$ using \eqref{kr} and \eqref{cr}.  These numbers are given and fixed.

Given positive integers $t$ and $r$ with $1 \le t \le r,$  
 we must prove, using \eqref{rectangle}, that \eqref{rt1neq0} holds and that
\eqref{lt} and \eqref{lt2} are satisfied using the function presented in \eqref{ltkr}.
To further clarify this, note that in Section 1, we defined $l(t)$ in terms of its properties,
\eqref{lt} and \eqref{lt2}; it is the rectangle-column
position of the last non-zero entry in the $t$-th row. What we must prove 
is that for each $t,$ the functional form for $l(t)$ presented in \eqref{ltkr}  
has these defined properties.

 The proof will be by induction on $t.$ 

For a base case, let $t=1.$  By \eqref{ltkr}, $l(1)=2.$ 
Equations \eqref{rt1neq0}, \eqref{lt2}, and \eqref{lt}  
then follow from \eqref{firstrow}.

Clearly, if $r=1,$ we are done. Therefore for the rest of the proof we assume $r \ge 2.$
  
Further assume, using an induction assumption, that
\eqref{lt}, \eqref{lt2},  and \eqref{rt1neq0} 
using \eqref{ltkr}   are true for  
the cases $1,2,\dotsc,t -1$ for some $t$ satisfying 
\begin{equation}\label{2tr}
	2 \le t \le r.
\end{equation}
 To complete the proof we must using \eqref{ltkr} for the case $t,$ show that
\eqref{rt1neq0}, \eqref{lt2} and \eqref{lt} all hold for the case $t.$ 

\textbf{Proof of \eqref{lt} using \eqref{ltkr}.}
First note that by \eqref{ltkr}
\begin{equation}\label{ltt-1}
	l(t) = 2 + (t-1)q = l(t-1)+q.
\end{equation}
Let $i=t$ and $j=l(t)$ in \eqref{trianglerecursion}. Then by \eqref{ltt-1}, the last summand on the 
right-hand side  of \eqref{trianglerecursion} satisfies 
$-R_{i-1,j-q} = -R_{t-1,l(t)-q} = -R_{t-1,l(t-1)} \neq 0$ 
the last inequality following from the induction  assumption.

To complete the proof of \eqref{lt}, it suffices to show 
that the remaining summands on the right-hand side of \eqref{trianglerecursion}
 are equal to 0. This follows immediately from the induction assumption on \eqref{lt2} and by the observation
that $l(t-1)+q < c(r)$ which follows from \eqref{2tr}, \eqref{ltt-1}, and \eqref{cr}.

\textbf{Proof of \eqref{lt2}.}  Let $u$ satisfy 
\begin{equation}\label{B}
	l(t)+1 \le u \le c(r).
\end{equation}
Let $i=t$
and $j=u$ in \eqref{trianglerecursion}. We claim that all $q+1$ summands on the right-hand
side of \eqref{trianglerecursion} are 0. Indeed, these $q+1$ summands contain the following factors:
$$
	\{R_{t-1,k}: u-q \le k \le u\} \subset   \{R_{t-1,k}: 1+l(t-1) \le k \le c(r)\},
$$
the set inclusion following from \eqref{B} and \eqref{ltt-1}. 
The proof of \eqref{lt2} is completed by  the induction assumption on \eqref{lt2}.

\textbf{Proof \eqref{rt1neq0}}. In \eqref{trianglerecursion}, let $i=t$ and $j=1.$ 
 Then the first summand on the right-hand side of \eqref{trianglerecursion} satisfies 
$-(a_q-1)R_{t-1,1} \neq 0,$ the inequality following from \eqref{qap} and  
the induction assumption on \eqref{rt1neq0}.

To complete the proof of \eqref{rt1neq0} it suffices to show the remaining $q$ summands
on the right-hand side of \eqref{trianglerecursion} are 0.
 
By \eqref{boundary}, these $q$ summands contain factors
\begin{equation}\label{A}
	 \{R_{t-1,1-u}\}_{1 \le u \le q}  = \{R_{t-2,c+1-u}\}_{1 \le u \le q}.
\end{equation}
  
We claim these factors are all 0. To prove this we consider two cases.

\textbf{Case $t=2.$} 
That all rectangle elements on the right-hand side of \eqref{A} are zero follows
from \eqref{ivalues} and the observation that $2 \le c-q+1$ which follows from
\eqref{cr} and \eqref{2tr}.

\textbf{Case $t \ge 3.$}  
The fact that the rectangle elements 
 the right-hand side of \eqref{A} are 
all zero follows from our induction assumption on \eqref{lt2} for the case $t-2,$ and
the observation that  $l(t-2) +1 \le c-q+1$ which follows from \eqref{2tr}, \eqref{ltkr},
and \eqref{cr}.

The proof of the main theorem is complete.  Some descriptive corollaries are easy consequences.

\begin{corollary} With notations as above, for $1 \le t \le r,$
$$R_{t,1} = \biggl( - (a_q-1)\biggr)^{(t-1)}, \qquad R_{t,l(t)} = (-1)^t.$$
\end{corollary}
\begin{proof} Equation \eqref{firstrow} shows this corollary true for $t=1.$
Inspecting the proof presented above shows that for $t \ge 2,$ $R_{t,1} = R_{t-1,1}\times -(a_q-1)$ and 
$R_{t,l(t)} = R_{t-1,l(t-1)} \times -1.$
\end{proof}

\begin{corollary}
The sequence of the positions of the last non-zero element in the  rows of 
the rectangle appearing embedded in the recursive sequence of order $k(r)$ 
is an arithmetic sequence with common difference, $q.$
 $$\langle l(1), l(2),  \dotsc, l(r) \rangle = \langle 2, 2+q, 2+2q, \dotsc, 2+(r-1)q = c \rangle $$
\end{corollary}
 
Table 3 nicely numerically illustrates both corollaries.

\section{conclusion}

In concluding the paper, we mention some frequently asked questions. The idea
of triangles embedded in recursions was first presented orally  at the 
West Coast Number Theory conference in 2017 during one of the problem sessions. 
The theory as presented in this paper evolved over the past three years at several
conferences including the Sarajevo conference on Fibonacci numbers and their applications.
Understandably, the same questions tend to be repeated. We single out two frequently 
asked questions.

People frequently ask what would happen if there were no $a_i, 1 \le i \le q,$ \eqref{qap}.
It turns out this was the initial case presented orally (as a sketch) in 2017. It was not
included in the main theorem to avoid introducing several necessary distinctions.

But the main results can be summarized as follows. First \eqref{kr} must be replaced by $k(r)=1+r,$ $r=1,2,3, \dotsc.$
Equations \eqref{cr} and \eqref{ltkr} are replaced by $c(r)=2+(r-1), l(t) = 2+(r-1).$
The main theorem can then be formulated in terms of the family of recursive sequences
associated with the Taylor series $T(X)= 2+ \frac{1}{1-X}.$ The main theorem would state
that for each $r,$ the associated recursion of order $k(r)$ appears to have a triangle 
(or an $r \times c(r)$ rectangle) embedded in it. Here, the associated recursion of order $k(r)$
is, as in this paper, the recursion corresponding to the approximating polynomial of $T(X),$ with initial values
given by \eqref{ivalues}. The corresponding triangle recursion is 
$R_{t,s}=2R_{t-1,s}-R_{t-1,s-1}.$ For $r=3$ the first three rows laid out in a 
$3 \times 4$ rectangle are presented in Figure 4.

\begin{center}
\begin{table}
\begin{tabular}{||c|c c c c ||}
\hline \hline
$\text{Position}$		&	1	&	2	&	3	&	4\\
\hline
$R_{1,1}, \dotsc, R_{1,c}$	&	1	&	-1	&	0	&	0\\
$R_{2,1}, \dotsc, R_{2,c}$	&	2	&	-3	&	1	&	0\\
$R_{3,1}, \dotsc, R_{3,c}$	&	4	&	-8	&	5	&	-1\\
\hline \hline
\end{tabular}
\begin{caption}{ $2 \times 4$ rectangle for the order $k(3)=4$ recursion} \end{caption}
\end{table}
\end{center}

A second frequently question asked is whether there are examples of the general theory
with a Taylor series not involving a geometric series. This is an open question. One
can take any Taylor series based on a rational function and develop a family of recursive sequences and view the
rectangles that appear embedded in them which do obey the triangle recursion presented
 in Proposition 1.4 (with appropriate adjustments for the denominator of the rational function).
The catch is there are no triangles in these rectangles because the delimiting zeroes are
not present. 

This observation illuminates the main theorem presented in this paper. The coefficients 
in the Taylor series $T(X)$ must avoid any rapid increase so that the associated recursive
sequences have sufficient zeroes in them. Thus, it remains an open problem to find examples
 of the theory presented in this paper based on a Taylor series without a geometric series
component.

\medskip

\noindent MSC2010: 11B39, 33C05

\end{document}